%% file: analco_final_final.tex
\documentclass[twoside,leqno,twocolumn]{article}

\usepackage{amsmath,amssymb,dsfont,mathtools,xfrac}
\usepackage{hyperref,listings,placeins}
\usepackage[capitalize,compress]{cleveref}
\usepackage[algo2e,ruled,vlined]{algorithm2e}
\makeatletter
\def\convx{\@ifnextchar[{\@with}{\@without}}
\def\@with[#1]#2{\xrightarrow[#1]{#2}}
\def\@without#1{\xrightarrow[n\to\infty]{#1}}
\def\blfootnote{\xdef\@thefnmark{}\@footnotetext}
\makeatother

\newcommand{\wrt}{w.r.t.\ }

\newcommand{\dd}{\ensuremath{\mathrm{d}}}
\newcommand{\Pb}{\ensuremath{\mathbb{P}}}
\newcommand{\I}{\ensuremath{\mathds{1}}}

\newcommand{\R}{\ensuremath{\mathbb{R}}}
\newcommand{\Prob}{\ensuremath{\mathbb{P}}}
\newcommand{\bo}{\ensuremath{\mathrm{O}}}
\newcommand{\Var}{\ensuremath{\mathrm{Var}}}
\newcommand{\Erw}[1]{\ensuremath{\mathbb{E}\!\left[#1\right]}}
\newcommand{\E}{\ensuremath{\mathbb{E}}}
\newcommand{\eqd}{\ensuremath{\stackrel{d}{=}}}

\renewcommand{\le}{\leqslant}

\crefname{figure}{Figure}{Figures}
\crefname{algocf}{Algorithm}{Algorithms}
\crefname{@theorem}{Theorem}{Theorems}
\providecommand\dontprintsemicolon{\DontPrintSemicolon}

\newif\ifquantilefancyformat
\quantilefancyformattrue

\usepackage{ltexpprt}

\title{A statistical view on exchanges in Quickselect}

\author{
    Benjamin Dadoun\\
    D\'epartement Informatique\\
    \'Ecole Normale Sup\'erieure de Cachan\\
    94235 Cachan Cedex\\
    France\\
    Email: {\tt \small benjamin.dadoun@ens-cachan.fr}
 \and
    Ralph Neininger\\
    Institute for Mathematics\\
    \vphantom{\'E}J.W.~Goethe University\\
    60054 Frankfurt a.M.\\
    Germany\\
    Email: {\tt \small neiningr@math.uni-frankfurt.de}
}

\date{}

\begin{document}

\maketitle
\renewcommand\footnoterule{\kern-3pt \hrule width 1in \kern 2.6pt}
\blfootnote{This research was done during an internship of the first
mentioned author at J.W.~Goethe University from June 2013 to August
2013.}

\begin{abstract}\small\baselineskip=9pt
In this paper we study the number of key exchanges
required by Hoare's FIND algorithm (also called Quick\-select) when
operating on a uniformly distributed random permutation and selecting an
independent uniformly distributed rank. After normalization we give a
limit theorem where the limit law is a perpetuity characterized by a
recursive distributional equation. To make the limit theorem usable
for statistical methods and statistical experiments we provide an
explicit rate of convergence in the Kolmogorov--Smirnov metric,
a numerical table of the limit law's distribution function and an
algorithm for exact simulation from the limit distribution. We also
investigate the limit law's density. This case study provides a program
applicable to other cost measures, alternative  models for the rank  selected  and more balanced choices of the pivot element such as
median-of-$2t+1$ versions of Quickselect as well as further variations
of the algorithm.
\end{abstract}

\medskip
\noindent
{\textbf{MSC2010:} 60F05, 68P10, 60C05, 68Q25.


\section{Introduction}
For selecting ranks within a finite list of data from  an ordered
set, Hoare \cite{ho61} introduced the algorithm FIND, also called
Quick\-select, which is a one sided version of his sorting algorithm
Quicksort. The data set is partitioned into two sub-lists by use of a
pivot element, then the algorithm is recursively applied to the sub-list
that contains the rank to be selected, unless its size is one. Hoare's
partitioning procedure is performed by scanning the list with pointers
from left and right until misplaced elements are found. They are
flipped, what we count as one key exchange. This scanning step is
then further performed until the pointers meet within the list. For
definiteness, in this paper we consider the version of
Hoare's partitioning procedure presented in Cormen,  Leiserson and
Rivest \cite[Section 8.1]{co90}. (However, our asymptotic results are
robust to small changes in the partitioning procedure, e.g.~they also
hold for the versions of Hoare's partitioning procedure described in
Sedgewick \cite[p.~118]{se90} or Mahmoud \cite[Exercise 7.2]{ma00}.)

We consider the probabilistic model where $n$ distinct data are given in
uniformly random order and where the rank to be selected is uniformly
distributed on $\{1,\ldots,n\}$ and independent of the permutation of
the data. In this model
the number of key comparisons has been studied in detail in Mahmoud,
Moddares and Smythe \cite{momosm95}. For the number
 $Y_n$  of key exchanges the mean has been identified exactly by means
 of analytic combinatorics: In Mahmoud  \cite{ma09},  for the number of
 data moves $M_n$ which is essentially (the partitioning procedure used
 in \cite{ma09} being slightly different to ours) twice our number of
 key exchanges it is shown that
 \begin{align}\label{mean_mov}
 \E[M_n]=n +\frac{2}{3}H_n -\frac{17}{9} +\frac{2H_n}{3n} -
 \frac{2}{9n}.
 \end{align}
Note that lower order terms here depend on the particular version of
Hoare's partitioning procedure used.
Moreover, for the variance, Mahmoud  \cite{ma09} obtained, as $n\to
\infty$ that
\begin{align}\label{var_mov}
\frac{1}{15}n^2 + \bo(n) \leqslant  \Var(M_n)\leqslant \frac{41}{15}n^2
+ \bo(n),
 \end{align}
where the Bachmann--Landau $\bo$-notation is used. A different
partitioning procedure due to Lomuto is analyzed in  Mahmoud
\cite{ma10}. Key exchanges in related but different models are studied in \cite{hwts02,mapapr11}. In the present paper we extend the
analysis started in \cite{ma09} of Quickselect with Hoare's partition
procedure. Together with more refined results stated below we identify
the asymptotic order of the variance and provide a limit law:
\begin{theorem}\label{rn_thm1}
For the number $Y_n$ of key exchanges used by Hoare's Quickselect
algorithm when acting on a uniformly random permutation of size $n$ and
selecting an independent uniform rank we have, as $n\to \infty$, that
 \begin{align}\label{rn_ll}
 \frac{Y_n}{n} \stackrel{d}{\longrightarrow} X,
 \end{align}
where the distribution of $X$ is the unique solution of the
recursive distributional equation
\begin{align}\label{rn_fix}
 X \eqd \sqrt{U}X + \sqrt{U}(1-\sqrt{U}),
 \end{align}
 where $X$ and $U$ are independent and $U$ is uniformly distributed
 on $[0,1]$.\\
 Moreover, we have $\Var(Y_n) \sim \frac{1}{60} n^2$ as $n\to \infty$.
 \end{theorem}
\cref{rn_thm1} follows quite directly from the contraction
method and is a corollary to  more refined convergence results in our
\cref{rn_thm33,rn_thm35}. We also obtain $\Var(M_n)\sim \frac{1}{15}n^2$
and $M_n/n \to 2X$ in distribution as  $n\to \infty$,
cf.~\eqref{var_mov}. An interpretation of the coefficients $\sqrt{U}$
and $\sqrt{U}(1-\sqrt{U})$ appearing
in \eqref{rn_fix} is given in \cref{rn_rem} below.

Recursive distributional equations such as \eqref{rn_fix} appear
frequently in the asymptotic analysis of random tree models and of
complexities of recursive algorithms; they also appear in insurance
mathematics as so-called perpetuities and in probabilistic number
theory. It should be noted that solutions of recursive distributional
equations are typically difficult to access, e.g., with respect to
their density if a density exists.

Recall that the original purpose of a limit law, such as our limit law
\eqref{rn_ll}, consists of being able to approximate the distributions
of $Y_n$ by their scaled limit $X$. However, such an approximation can
only be made effective  if characteristics of the distribution ${\cal
L}(X)$ of $X$ are accessible and the distance between ${\cal L}(X)$
and ${\cal L}(Y_n/n)$ can be bounded explicitly.

For this reason  we take a statistician's point of view: To make
the limit theorem \eqref{rn_ll} usable for statistical methods
and statistical experiments we provide an explicit bound on the
rate of convergence in the Kolmogorov--Smirnov metric in
\cref{sec_conv} (\cref{rn_thm35}), a numerical table of the
distribution function of ${\cal L}(X)$ in \cref{sec_dens}
(\cref{fig:cdf}) and an algorithm for exact simulation from
${\cal L}(X)$ in \cref{sec_perf} (\cref{alg:sampling}).
The density and further properties of ${\cal L}(X)$ are studied in
\cref{sec_dens}. In \cref{sec_2} the recursive approach our analysis is
based on is introduced together with some combinatorial preliminaries.

We consider this paper as a case study with a program
applicable to other cost measures, alternative  models for the rank
selected  and more balanced choices of the pivot element such as
median-of-$2t+1$ versions of Quickselect as well as further variations
of the algorithm.

\section{Distributional recurrence and preliminaries}\label{sec_2}
The first call to the partitioning procedure (in the version
\cite[Section 8.1]{co90} we consider here) splits the given uniformly
distributed list $A[1..n]$ of size $n$ into two sub-lists of sizes $I_n$ and $n-I_n$ as follows: The first element $p:=A[1]$
is chosen as the pivot element, and the list is scanned both forwards and
backwards with two indices $i$ and $j$, looking for elements with
$A[i]\geqslant p$ and  elements with $A[j]\leqslant p$. Every
misplaced pair $(A[i],A[j])$ found is then flipped, unless $i$ has
become greater than or equal to $j$, where we stop (resulting in $I_n=j$).
Note that the pivot element is moved to the right sub-list if there is
at least one key exchange, and the event $\{I_n=1\}$ thus occurs if and
only if the leftmost element in the array is the smallest or the second
smallest element of the whole array. Together with the  uniformity
assumption we obtain
 \begin{align*}
 \Prob(I_n=1)=\frac{2}{n}, \quad
 \Prob(I_n=j)=\frac{1}{n} \mbox{ for } j=2,\ldots,n-1.
 \end{align*}
  The list with elements with value less or equal to the pivot element
  we call the {\em left sub-list}, its size is $I_n$, the other list
  we call the {\em right sub-list}. We denote by $T_n$ the number of
  key exchanges executed during the first call to the partitioning
  procedure. Note that $T_n$ is random and that $I_n$ and $T_n$ are
  stochastically dependent (for all $n$ sufficiently large). Since
  during the first partitioning step comparisons are only done between
  the elements and the pivot element we have that conditional on the
  size $I_n$ the left and right sub-list are uniformly distributed
  and independent of each other. The number $Y_n$ of key exchanges
  (key swaps) required by Quickselect  (when operating on a uniformly
  permuted list of size $n$ of distinct elements and selecting a rank
  $R_n$ uniformly distributed over $\{1,\ldots, n\}$ and independent of
  the list) allows a recursive decomposition. The recursive structure
  of the algorithm, the properties of the partitioning procedure
  and the model for the rank to be selected imply  $Y_1=0$ and, for $n\geqslant 2$, the
  distributional recurrence
  \begin{align}\label{rn_rec_1}
  Y_n\eqd \I_{\{R_n \leqslant I_n\}} Y_{I_n} +
  \I_{\{R_n > I_n\}} Y'_{n-I_n}+T_n.
  \end{align}
  Here $(Y'_j)_{1\leqslant j\leqslant n-1}$ is identically
  distributed as $(Y_j)_{1\leqslant j\leqslant n-1}$ and we have that
  $(Y_j)_{1\leqslant j\leqslant n-1}$, $(Y'_j)_{1\leqslant j\leqslant
  n-1}$ and $(I_n,T_n)$ are independent. To make the right hand side
  of the latter display more explicit we observe that the conditional
  distribution of $T_n$ given $I_n$ is hypergeometric:
  \begin{lemma} Conditional on $I_n=1$ the number $T_n$ of
  swaps during the first call to the partitioning procedure has
  the Bernoulli Ber$(\frac{1}{2})$ distribution. Conditional on
  $I_n=j$ for $j\in\{2,\ldots,n-1\}$ the random variable $T_n$ is
  hyper\-geo\-metrically Hyp$(n-1;j,n-j)$ distributed, i.e.,
\begin{align*}
\Prob(T_n=k\mid I_n=j)
        =\frac{\binom{j}{k}\binom{n-j-1}{n-j-k}}%
          {\binom{n-1}{n-j}},
        \end{align*}
for $\min(1,j-1)\leqslant k\leqslant\min(j,n-j)$.
    \end{lemma}
    \begin{proof}
    For simplicity of presentation we identify the elements of the array
    with their ranks, i.e., we assume the elements are $1,\ldots,n$
    in uniformly random order.
  Conditional on $I_n=1$ the leftmost element of the array is $1
  $ or $2$ resulting in respectively $0$ or $1$ key exchanges. The
  uniformity of the array implies the Ber$(\frac{1}{2})$ distribution
  in the statement of the Lemma. Conditional on $I_n=j$ with
  $j\in\{2,\ldots,n-1\}$ the pivot element $p$ is moved to the right
  sub-list and we have $I_n=p-1$.
  Thus, we have to
  count the number of permutations $\sigma$ of length $n$
  such that  $\sigma(i)\leqslant I_n$ for exactly $k$ indices
  $i\in\{I_n+1,\ldots,n\}$, among those having $p=I_n+1$ as first
  element. This implies the assertion.
  \end{proof}
The asymptotic joint behavior of $(I_n,T_n)$ will be crucial in our subsequent
analysis:
\begin{lemma}  \label{rn_con_coe}
For any $1\leqslant p<\infty$  we have
\begin{align*}
\left( \frac{I_n}{n},\frac{T_n}{n}\right)
\stackrel{\ell_p}{\longrightarrow} (U,U(1-U)) \qquad (n\to\infty),
\end{align*}
where $U$ has the uniform distribution on the unit interval $[0,1]$.
\end{lemma}
The convergence in $\ell_p$ (defined below) is equivalent to weak
convergence plus convergence of the $p$-th absolute moments.
\cref{rn_con_coe} follows below from \cref{rn_cora}.
The scalings in \cref{rn_con_coe} motivate the normalization
\begin{align}\label{rn_scal}
X_n:=\frac{Y_n}{n},\quad n\geqslant 1.
\end{align}
Recurrence \eqref{rn_rec_1} implies the distributional recurrence
\begin{equation}\label{re_rec_3}
  \begin{aligned}
  X_n&\eqd%
  \I_{\{\frac{R_n}n\leqslant\frac{I_n}n\}}\frac{I_n}nX_{I_n}\\
    &\qquad+\I_{\{\frac{R_n}n>\frac{I_n}n\}}\frac{n-I_n}nX'_{n-I_n}
    +\frac{T_n}n,
  \end{aligned}
\end{equation}
(for $n\geqslant2$) where, similarly to \eqref{rn_rec_1},
$(X'_j)_{1\leqslant j\leqslant n-1}$ is identically distributed as
$(X_j)_{1\leqslant j\leqslant n-1}$ and we have that
$(X_j)_{1\leqslant j\leqslant n-1}$,
$(X'_j)_{1\leqslant j\leqslant n-1}$ and $(I_n,T_n)$ are independent.

The asymptotics of \cref{rn_con_coe} suggest that a limit $X$
of $X_n$ satisfies the recursive distributional equation (RDE)
 \begin{equation}\label{rn_rde_1}
    \begin{aligned}
    X&\eqd\I_{\{V\leqslant U\}}UX\\
      &\qquad+\I_{\{V>U\}}(1-U)X'
      +U(1-U),
    \end{aligned}
  \end{equation}
  where $U,V,X,X'$ are independent,
  $U$ and $V$ are uniformly distributed on $[0,1]$ and $X'$ has the
  same distribution as $X$.
  \begin{lemma}\label{rn_lem23}
RDE \eqref{rn_rde_1} has a unique solution among all probability
distributions on the real line. This solution is also the unique
solution (among all probability distributions on the real line) of
RDE \eqref{rn_fix}.
  \end{lemma}
  \begin{proof}
A criterion of Vervaat \cite{ve79} states that a RDE of the form
$X=_d AX+b$ with $X$ and $(A,b)$ independent has a unique
solution  among all probability distributions
on the real line if $-\infty\leqslant \E[\log|A|]<0$ and
$\E[\log^+|b|]<\infty$. These two conditions are
satisfied for our RDE \eqref{rn_fix}. The full claim
of the Lemma hence follows by showing that the solutions
of RDE \eqref{rn_rde_1} are exactly the solutions
of RDE \eqref{rn_fix}. This can be seen using
characteristic functions as follows: Let ${\cal L}(Z)$
be a solution of RDE \eqref{rn_rde_1} and denote
its characteristic function by $\varphi_Z(t):=\E[e^{\mathrm{i}tZ}]$
for $t\in\R$. Conditioning on $U$ and $V$ and using independence we
obtain that
 \begin{align*}
    \varphi_Z(t) =\int_0^1 2u\varphi_Z(tu)e^{\mathrm{i}tu(1-u)}\dd
    u,\quad t\in\R.
  \end{align*}
Now, for the random variable $Y:=\sqrt{U}Z
      +\sqrt{U}(1-\sqrt{U})$, where $U$ is uniformly distributed on
      $[0,1]$ and independent of $Z$ we find that its characteristic
      function $\varphi_Y$ satisfies
 \begin{align*}
    \varphi_Y(t)
      &=
        \int_0^1\varphi_Z(t\sqrt{u})
        e^{\mathrm{i}t\sqrt{u}(1-\sqrt{u})}
        {\dd u}\\
      &=\int_0^12u\varphi_Z(tu)e^{\mathrm{i}tu(1-u)}\dd u =\varphi_Z(t),
      \quad t\in \R.
  \end{align*}
  This implies that ${\cal L}(Z)$ is a solution of RDE
  \eqref{rn_fix}. The same argument shows that every solution of RDE
  \eqref{rn_fix} is a solution of RDE \eqref{rn_rde_1}.
 \end{proof}

\begin{Remark}\label{rn_rem}
 Alternatively to recurrence \eqref{rn_rec_1} we have the recurrence
  \begin{align}\label{rn_rec_1alt}
  Y_n\eqd  Y_{J_n} +T_n,\quad n\geqslant 2,
  \end{align}
  with conditions as in \eqref{rn_rec_1} and $J_n$ denoting the size of
  the sub-list where the Quickselect algorithm recurses on. Note that
  by the uniformity of the rank to be selected $J_n$ is a size-biased
  version of $I_n$. Hence the limit (in distribution) of $J_n/n$
  is the size-biased version of the limit $U$ of $I_n/n$. Since
  $\sqrt{U}$ is a size-biased version of $U$, it appears in the
  RDE \eqref{rn_fix}. Moreover, the asymptotic joint behavior
  of $(J_n,T_n)$ is again determined by the concentration of the
  hypergeometric distribution as in \cref{rn_con_coe} (cf.~the proof of
  Lemma \ref{rn_cora}.) Analogously, we obtain $(J_n/n,T_n/n)\to
  (\sqrt{U},\sqrt{U}(1-\sqrt{U}))$ which explains the occurrences of the
  additive term $\sqrt{U}(1-\sqrt{U})$ in RDE \eqref{rn_fix}. (Note that
  this does not contradict Lemma \ref{rn_con_coe}, since $U(1-U)$ and
  $\sqrt{U}(1-\sqrt{U})$ are identically distributed.)
  We could as well base our subsequent analysis on \eqref{rn_rec_1alt}
  but prefer to work with recurrence \eqref{rn_rec_1}.
 \end{Remark}

\section{Convergence and rates}\label{sec_conv}
 In this section we bound the rate of convergence in the limit law of
 \cref{rn_thm1}. First, bounds in the minimal $\ell_p$-metrics
 are derived. These imply bounds on the rate of convergence within the
 Kolmogorov--Smirnov metric. For $1\leqslant p< \infty$ and probability
 distibutions ${\cal L}(W)$ and ${\cal L}(Z)$ with $\E[|W|^p]$,
 $\E[|Z|^p]<\infty$ the $\ell_p$-distance is defined by
 \begin{align*}
 \ell_p({\cal L}(W),{\cal L}(Z))&:= \ell_p(W,Z)\\
   &:= \inf\{\|W'-Z'\|_p \,|\,W'\eqd W,  Z'\eqd Z\}.
 \end{align*}
 The infimum is over all vectors $(W',Z')$ on a common probability
 space with the marginals of $W$ and $Z$. The infimum
 is a minimum and such a minimizing pair $(W',Z')$ is called an optimal
 coupling of ${\cal L}(W)$ and ${\cal L}(Z)$. For a sequence of random
 variables  $(W_n)_{n\geqslant 1}$ and $W$ we have, as $n\to \infty$,
 that
  \begin{align*}
\ell_p(W_n,W)\to 0 \quad \iff \quad \left\{ \begin{array}{l}
W_n \stackrel{d}{\longrightarrow} W, \\[2mm]
\E[|W_n|^p] \to \E[|W|^p].\end{array} \right.
 \end{align*}
 For these and further properties of $\ell_p$ see Bickel and Freedman
 \cite[Section 8]{bifr81}.

We start bounding
the rate in the convergence in \cref{rn_con_coe}. This can be done
using a tail estimate for the hypergeometric distribution derived in
Serfling \cite[Theorem 3.1]{se74}, restated here in a slightly weaker
form more convenient for our analysis:
\begin{lemma}\label{rn_lem_ser}
  Let $n\geqslant 2$, $j\in\{1,\ldots,n-1\}$ and $T_n^{(j)}$ be a
  random variable with
  hypergeometric distribution
  $\text{Hyp}(n-1;j,n-j)$.
  Then for all $p>0$ we have
  $$\Erw{\left|\frac{T_n^{(j)}}n-\frac{j(n-j)}{n(n-1)}\right|^p
    }\leqslant\frac{\Gamma(\sfrac{p}2+1)}{2^{\sfrac{p}2+1}}\;
      n^{-\sfrac{p}2},$$
  where $\Gamma$ denotes Euler's gamma function.
\end{lemma}

\begin{lemma}\label{rn_cora}
  For the number $T_n$ of key exchanges in the first  call to the
  partitioning procedure of Hoare's Quickselect we have for all
  $n\geqslant 2$ and all $1\leqslant p<\infty$ that
\begin{gather*}
\ell_p\left(\frac{T_n}n,U(1-U)\right)
    \leqslant(2+\tau_p)n^{-\sfrac12},\\\qquad\mbox{ where }
    \tau_p:=\left(\frac12+\frac{\Gamma(\sfrac{p}2+1)}%
      {2^{\sfrac{p}2+1}}\right)^{\sfrac1p}.
      \end{gather*}
\end{lemma}
\begin{proof}
Let $U$ be uniformly distributed over $[0,1]$ and the underlying
probability space sufficiently large so that we can also embed
the vector $(I_n,T_n)$ such that  $I_n= \lfloor nU \rfloor
+\I_{\{U\leqslant 1/n\}}$.
  Let $h(u):=u(1-u)$. The mean value theorem and
  $|\frac\dd{\dd u}h(u)|=|1-2u|\leqslant1$ for all $u\in[0,1]$ imply
  \begin{flalign*}
    &\left\|U(1-U)-\frac{I_n(n-I_n)}{n^2}\right\|_p^p\\
    &\qquad\begin{aligned}
    &=\sum_{k=0}^{n-1}\int_{\frac{k}n}^{\frac{k+1}n}
      \left|\textstyle
        h(u)-h\left(\frac{k\vee1}n\right)\right|^p\dd u\\
    &\leqslant\frac1{n^p}.
    \end{aligned}
  \end{flalign*}
We have $(\frac1{n-1}-\frac1n)\|I_n(n-I_n)/n\|_p\leqslant\frac1n$
since $1\leqslant I_n\leqslant n-1$ a.s.
Using \cref{rn_lem_ser} we obtain
\begin{flalign*}
  &\left\|\frac{T_n}n-\frac{I_n(n-I_n)}{n(n-1)}\right\|_p^p\\
  &\qquad\begin{aligned}
  &=\sum_{j=1}^{n-1}
    \Erw{\left|\frac{T_n}n-\frac{j(n-j)}{n(n-1)}\right|^p
      \;\middle|\;I_n=j}\Pb(I_n=j)\\
  &\leqslant\frac1{2n^p}+\frac{n-2}n\times
    \frac{\Gamma(\sfrac{p}2+1)}{2^{\sfrac{p}2+1}}\;n^{-\sfrac{p}2}\\
  &\leqslant\tau_p^p n^{-\sfrac{p}2}.
  \end{aligned}
\end{flalign*}
The triangle inequality implies
\begin{equation}\label{rn_bou_ref}
 \begin{aligned}
 \ell_p\left(\frac{T_n}n,U(1-U)\right)&\leqslant
 \left\|\frac{T_n}n-U(1-U)\right\|_p\\
  &\leqslant(2+\tau_p)n^{-\sfrac12},
 \end{aligned}
\end{equation}
the assertion.
\end{proof}
We obtain the following bounds on the rate of convergence in
\cref{rn_thm1}. For the proof of \cref{rn_thm33}
standard estimates from the contraction method, see
\cite{Ro91,RaRu95,RoRu01,NeRu04}, are applied.
\begin{theorem}\label{rn_thm33}
 For $Y_n$ and $X$ as in \cref{rn_thm1} we have for all
 $n\geqslant 1$ and all  $1\leqslant p<\infty$, that
 \begin{equation*}
   \ell_p\left(\frac{Y_n}{n},X\right)\leqslant\kappa_p n^{-\sfrac12},
   \quad \kappa_p:=\frac{2p+3}{2p-1}(7+\tau_p).
 \end{equation*}%
\end{theorem}
\begin{proof}
With $X_n$ as defined in \eqref{rn_scal} recall  the recurrence
\eqref{re_rec_3}:
\begin{equation}\label{rn_rec_3b}
  \begin{aligned}
  X_n&\eqd%
  \I_{\{\frac{R_n}n\leqslant\frac{I_n}n\}}\frac{I_n}nX_{I_n}\\
    &\qquad+\I_{\{\frac{R_n}n>\frac{I_n}n\}}\frac{n-I_n}nX'_{n-I_n}
    +\frac{T_n}n,
  \end{aligned}
\end{equation}
  For $X$ as in \cref{rn_thm1} we have, by \cref{rn_lem23},
  that
   \begin{equation}\label{rn_rde_1b}
    \begin{aligned}
    X&\eqd\I_{\{V\leqslant U\}}UX\\
      &\qquad+\I_{\{V>U\}}(1-U)X'
      +U(1-U),%
    \end{aligned}
  \end{equation}
  with conditions as in \eqref{rn_rde_1}.
Note that we can embed all random variables appearing on the right
hand sides of \eqref{rn_rec_3b} and \eqref{rn_rde_1b} on a common
probability space such that we additionally have that
 $I_n= \lfloor nU \rfloor + \I_{\{U\leqslant 1/n\}}$,  $R_n=\lceil
 nV\rceil$ and that $(X_j,X)$ and $(X'_j,X')$ are optimal couplings
 of ${\cal L}(X_j)$ and ${\cal L}(X)$ such that $(U,V)$, $(X_j,X)$,
 $(X'_j,X')$ for $j=1,\ldots,n-1$ are independent.

 Now, for $n\geqslant 2$ we define the random variable
 $$Q_n:=\I_{\{R_n\leqslant I_n\}}\frac{I_n}nX
    +\I_{\{R_n>I_n\}}\frac{n-I_n}nX'+\frac{T_n}n.$$
    The triangle inequality implies
\begin{align}\label{rn_tri}
\ell_p(X_n,X)\leqslant \ell_p(X_n,Q_n) + \ell_p(Q_n,X).
\end{align}
The second summand in \eqref{rn_tri} is bounded by
\begin{flalign*}
&\ell_p(Q_n,X)\\
&\begin{aligned}
    &\quad\leqslant\left\|\I_{\{V\leqslant U\}}U-\I_{\{R_n
      \leqslant I_n\}}\frac{I_n}n\right\|_p\\
      &\hspace{3em}+\left\|\I_{\{V>U\}}(1-U)-
      \I_{\{R_n>I_n\}}\frac{n-I_n}n\right\|_p\\
      &\hspace{3em}+\left\|\frac{T_n}n-U(1-U)\right\|_p\\
  &\quad\leqslant \frac{2}{n}+   \frac{2}{n} + \frac{2+\tau_p}{\sqrt{n}},
\end{aligned}
\end{flalign*}
 where we plug in the right hand sides of   \eqref{rn_rec_3b}
 and \eqref{rn_rde_1b}, use independence, that $\|X\|_p\leqslant
 1$ and the bound in \eqref{rn_bou_ref}. For the first summand in
 \eqref{rn_tri}  conditioning on $R_n$ and $I_n$ and using that
 $(X_j,X)$ and $(X'_j,X')$ are optimal couplings of ${\cal L}(X_j)$
 and ${\cal L}(X)$ we have
 \begin{align*}
 \ell_p(X_n,Q_n)\leqslant\frac1n\sum_{i=1}^{n-1}\frac{i^p(2i-1)
    +\I_{\{i=1\}}}{n^{p+1}}\ell_p(X_i,X).
    \end{align*}
  The summand $\I_{\{i=1\}}\ell_p(X_1,X)$  is
  bounded by $1$ since $X_1=0$ and $\|X\|_p\leqslant1$. Putting the
  estimates together we obtain
  \begin{equation*}
    \ell_p(X_n,X)\leqslant\frac1n\sum_{i=1}^{n-1}
    \frac{i^p(2i-1)}{n^{p+1}}\ell_p(X_i,X)+\frac{7+\tau_p}{\sqrt{n}}.
  \end{equation*}
  Now, by induction, we show $\ell_p(X_n,X)\leqslant
  \kappa_pn^{-\sfrac12}$. Since $\kappa_p\geqslant 1$ the assertion
  is true for $n=1$.
  For  $n\geqslant2$  using the induction hypothesis we obtain
  \begin{align*}
   \ell_p(X_n,X)&\leqslant\frac1n\sum_{i=1}^{n-1}
    \frac{i^p(2i-1)}{n^{p+1}}
    \frac{\kappa_p}{\sqrt{i}}+(7+\tau_p)n^{-\sfrac12}\\
    &\leqslant\frac{\kappa_p}{n^{p+2}}\sum_{i=1}^{n-1}\int_i^{i+1}
    2x^{p+\sfrac12}\dd x+(7+\tau_p)n^{-\sfrac12}\\
   &\leqslant\frac{\kappa_p}{n^{p+2}}\int_0^n2x^{p+\sfrac12}\dd x
     +(7+\tau_p)n^{-\sfrac12}\\
     &=\left[\frac4{2p+3}\kappa_p+(7+\tau_p)\right]
     n^{-\sfrac12}\\
     &=\kappa_pn^{-\sfrac12}.
  \end{align*}
 This finishes the proof.
\end{proof}

The
Kolmogorov--Smirnov distance between ${\cal L}(W)$ and ${\cal L}(Z)$
is defined by
\begin{align*}
d_{\mathrm{KS}}({\cal L}(W),{\cal L}(Z))&:=d_{\mathrm{KS}}(W,Z)\\
&:=\sup_{x\in\mathbb{R}}|\Prob(W\leqslant x)-\Prob(Z\leqslant x)|.
\end{align*}
Bounds for the $\ell_p$  distance can be used to bound $d_{\mathrm{KS}}$
using the following
lemma from Fill and Janson \cite[Lemma 5.1]{fija02}:
\begin{lemma} \label{lem_fija}
  Suppose that $W$ and $Z$ are two random variables such that $Z$
  has a bounded Lebesgue density $f_Z$. For all
  $1\leqslant p<\infty$, we have
  $$d_{\mathrm{KS}}(W,Z)\leqslant (p+1)^{\frac1{p+1}}
    \Bigl(\|f_Z\|_\infty \ell_p(W,Z)\Bigr)^{\frac{p}{p+1}}.$$%
\end{lemma}
Combining \cref{rn_thm33}, \cref{lem_fija} and
\cref{rn_den_bd} we obtain the following bound:
\begin{theorem}\label{rn_thm35}
 For $Y_n$ and $X$ as in \cref{rn_thm1} we have for all
 $0<\varepsilon\leqslant \frac{1}{4}$ and all $n\geqslant 1$ that
  \begin{equation*}
   \begin{aligned}
    d_{\mathrm{KS}}\left(\frac{Y_n}{n},X\right)&\leqslant
    \omega_\varepsilon n^{-\sfrac12+\varepsilon},\\
    \omega_\varepsilon&:=\left(\frac{1}{2\varepsilon}
      \right)^{2\varepsilon}\Big\{\|f\|_\infty
        \kappa_{-1+1/(2\varepsilon)}\Big\}^{1-2\varepsilon},
   \end{aligned}
  \end{equation*}
  where $f$ denotes the density of $X$.
\end{theorem}
\begin{proof}
To $0<\varepsilon\leqslant \frac{1}{4}$ choose
$p=-1+1/(2\varepsilon)$ in \cref{rn_thm33}.
\end{proof}

\begin{figure*}
  {\footnotesize\centering{\input{cdf}}}
  \caption{Distribution function of the solution of
    $X=_d\sqrt{U}X+\sqrt{U}(1-\sqrt{U})$, which is the limit
    distribution in \cref{rn_thm1}. All values are exact up to $10^{-4}$. The value at, e.g., $0.355$
    can be found in column labelled $0.3$ and row labelled $0.055$
    as $\Prob(X\leqslant 0.355)\approx 0.1376$. }%
  \label{fig:cdf}
\end{figure*}

\section{Density and distribution function}\label{sec_dens}
In this section we derive properties of the limit $X$ in
\cref{rn_thm1} mainly concerning its density and distribution
function. In particular, in \cref{rn_den_bd} we obtain a bound for the density $f$ of $X$ as required for Theorem \ref{rn_thm35}. Most results in this section are derived along the lines of
\cite[Section 5]{knne08}, where the related RDE
\begin{align}\label{rn_rde_2}
X\eqd UX + U(1-U),
 \end{align}
 discovered in Hwang and Tsai \cite{hwts02}, is studied.  We start with moments:
\begin{lemma}\label{rn_lemmom}
  For the limit $X$ in Theorem \ref{rn_thm1} for all $k\geqslant1$, we have
  \begin{equation*}
    \E[X^k]=2(k+2)!(k-1)!\sum_{i=0}^{k-1}\frac{\E[X^i]}{(2k-i+2)!i!}.%
  \end{equation*}
  In particular, $\E[X]=\frac12$, $\E[X^2]=\frac4{15}$ and
  $\mathrm{Var}(X)=\frac1{60}$.%
  \label{thm:moments}
\end{lemma}
\begin{proof}
 We raise left and right hand side of equation \eqref{rn_fix} to the power of
  $k$ and take expectations. This implies
  \begin{align*}
    \E[X^k]&=\Erw{\left(\sqrt{U}X+\sqrt{U}(1-\sqrt{U})\right)^k
      }\\
    &=\sum_{i=0}^k\binom{k}i\Erw{\sqrt{U}^k
      (1-\sqrt{U})^{k-i}}\E[X^i]\\
    &=2(k+1)!k!\sum_{i=0}^k\frac{\E[X^i]}{(2k-i+2)!i!},
  \end{align*}
  where we used that
  \begin{align*}
    \Erw{\sqrt{U}^k(1-\sqrt{U})^{k-i}}
    =\frac{2(k+1)!(k-i)!}{(2k-i+2)!}.
  \end{align*}
  This implies the assertion.
\end{proof}

\begin{lemma}\label{rn_tailb}
 For the limit $X$ in Theorem \ref{rn_thm1} we have $X\in[0,1]$ almost surely. For all $\varepsilon>0$ and all
 $k\geqslant1$,
  \begin{align}
    \Pb(X\geqslant1-\varepsilon)&\leqslant2^{\frac{k(k+3)}4}
      \varepsilon^{\frac{k}2}.%
    \label{eq:tail_probability}
  \end{align}%
\end{lemma}
\begin{proof} For the first claim set
  $Z_0:=0$ and
  $$Z_{n+1}:=\sqrt{U_{n+1}}Z_n+\sqrt{U_{n+1}}(1-\sqrt{U_{n+1}}),$$ where,
  for all $n\geqslant0$, $U_{n+1}$ is uniformly distributed on $[0,1]$ and
  independent of $Z_n$. This construction implies that
  $\Pb(Z_n\in[0,1])=1$ for all $n\geqslant0$. Since $Z_n$ tends to $X$ in law we obtain $\Pb(X\in[0,1])=1$.

  For the second claim note that (\ref{rn_fix}) implies
  $$\Pb(X\geqslant1-\varepsilon)=\Pb\left(\sqrt{U}X+\sqrt{U}(1-\sqrt{U})
    \geqslant1-\varepsilon\right).$$
  On the event $\{\sqrt{U}X+\sqrt{U}(1-\sqrt{U})\geqslant1-\varepsilon\}$ we  have
  \begin{align*}
      X&\geqslant2\sqrt{1-\varepsilon}-1 \quad \text{and}\\
    \sqrt{U}&\geqslant\frac{1+X-\sqrt{(1+X)^2-4(1-\varepsilon)}}2.
  \end{align*}
  Using that $0\leqslant X\leqslant1$ almost surely we obtain
  $X\geqslant1-2\varepsilon$, and
  $\sqrt{U}\geqslant\sqrt{1-\varepsilon}-\sqrt\varepsilon$, hence $U\geqslant1-2\sqrt\varepsilon$.
  By independence this implies
  \begin{align*}
    \Pb(X\geqslant1-\varepsilon)
      &\leqslant\Pb(X\geqslant1-2\varepsilon,
        U\geqslant1-2\sqrt\varepsilon)\\
      &=2\sqrt\varepsilon\Pb(X\geqslant1-2\varepsilon).
  \end{align*}
  Iterating the latter inequality $k\geqslant1$ times yields
  \begin{align*}
    \Pb(X\geqslant1-\varepsilon)
      &\leqslant(2\sqrt\varepsilon)^k\Pb(X\geqslant1-2^k\varepsilon)
        \sqrt2\sqrt4\cdots\sqrt{2^{k-1}}\\
      &\leqslant2^{\frac{k(k+3)}4}\varepsilon^{\frac{k}2}.
  \end{align*}
\end{proof}
We turn to the density of $X$:
\begin{theorem}\label{thm:density}
 The limit $X$  in Theorem \ref{rn_thm1} has a Lebesgue density $f$ satisfying $f(t)=0$ for
 $t<0$ or $t>1$, and, for $t\in[0,1]$,
  \begin{equation}\label{eq:density}
    \begin{aligned}
    f(t)&=2\int_{p_t}^tg(x,t)f(x)\dd x  \\
    &\qquad
    +\int_t^1(g(x,t)-1)f(x)\dd x,
    \end{aligned}
  \end{equation}
  where $p_t:=2\sqrt{t}-1$.\\
  Here, for $x\in[0,1]$ and $t<((1+x)/2)^2$,
  \begin{equation}\label{eq:density_aux2}
      g(x,t):=\frac{1+x}{\sqrt{(1+x)^2-4t}}.
  \end{equation}%
\end{theorem}
\begin{proof}%
  \label{prf:density}
 Let $\mu:={\cal L}(X)$
  denote the law of $X$ and $B\subset\R$ any Borel set. By (\ref{rn_fix}) we obtain
  \begin{align*}
    \Pb(X\in B)
    &=\Pb\left(\sqrt{U}X+\sqrt{U}(1-\sqrt{U})\in B\right)\notag\\
    &=\int_0^1\Pb\left(\sqrt{U}x+\sqrt{U}(1-\sqrt{U})
      \in B\right)\dd\mu(x)\notag\\
    &=\int_0^1\int_B\varphi(x,t)\dd t\,\dd\mu(x)\notag\\
    &=\int_B\left(\int_0^1\varphi(x,t)\dd\mu(x)\right)\dd t,
  \end{align*}
  where $\varphi(x,\cdot)$ denotes the Lebesgue
  density of
  $\sqrt{U}x+\sqrt{U}(1-\sqrt{U})$ for $x\in[0,1]$.
  Hence, $X$ has a Lebesgue density $f$ satisfying
  $f(t)=\int_0^1\varphi(x,t)\dd\mu(x)$, thus
  \begin{equation*}
    f(t)=\int_0^1\varphi(x,t)f(x)\dd x.%
  \end{equation*}
  It remains to identify $\varphi(x,\cdot)$: The distribution function  $F_x$ of
  $\sqrt{U}x+\sqrt{U}(1-\sqrt{U})$ is given by
  \begin{equation*}
    F_x(t)=\begin{cases}
    0,&\text{if $t<0$},\\
    \left(\frac{1+x-\sqrt{(1+x)^2-4t}}2\right)^2,
      &\text{if }t<x,\\
    1-(1+x)\sqrt{(1+x)^2-4t},
      &\text{if }t<\left(\frac{1+x}2\right)^2,\\
    1,&\text{otherwise}.
    \end{cases}%
  \end{equation*}
  Thus
  \begin{equation}
  \begin{aligned}
    \varphi(x,t)
    &=\begin{cases}
      2g(x,t),&\text{if $p_t<x\leqslant t$},\\
      g(x,t)-1,&\text{if $t<x\leqslant1$},\\
      0,&\text{otherwise,}
    \end{cases}
  \end{aligned}%
  \label{rn_den_sim}
  \end{equation}
  which implies the assertion.
\end{proof}

\begin{Remark}\label{rem:density_aux_m}
  Note that, \wrt $x$, $g$ defined in \eqref{eq:density_aux2}
   admits the simple primitive
  \begin{equation*}
    G(x,t)=\sqrt{(1+x)^2-4t}
  \end{equation*}
  which is $0$ at $x=p_t$. Moreover, $g(x,\cdot)$ is increasing for fixed $x$
  and $g(\cdot,t)$ is decreasing for fixed $t$. Indeed,
  \begin{align*}
    \frac{\partial g}{\partial t}(x,t)
    &=\frac12(1+x)\left((1+x)^2-4t\right)^{-\sfrac32}>0,
    \intertext{and}
    \frac{\partial g}{\partial x}(x,t)
    &=-4t\left((1+x)^2-4t\right)^{-\sfrac32}\leqslant0,
  \end{align*}%
  with equality if and only if $t=0$.
\end{Remark}
\begin{corollary}\label{cor:density_increasing}
The version of the density $f$ of $X$ with \eqref{eq:density} satisfies
$f(0)=0$, $f(1)=0$ and is increasing on $[0,\frac14]$.
\end{corollary}
\begin{proof}
  Since $f(x)=0$ for all $x\in(p_0,0)=(-1,0)$ and $g(x,0)=1$ for all
  $x\in(0,1)$, we obtain $f(0)=0$ from
  \eqref{eq:density}. Since $p_1=1$, we also get $f(1)=0$.

 For the monotonicity from \eqref{eq:density} we obtain, for $0\leqslant s<t\leqslant\frac14$, that
  \begin{align*}
    f(t)-f(s)&=
      \int_0^1\underbrace{[g(x,t)-g(x,s)]}_{>0}f(x)\dd x\\&\qquad+
      \int_0^s\underbrace{[g(x,t)-g(x,s)]}_{>0}f(x)\dd x\\&\qquad+
      \int_s^t\underbrace{[g(x,t)+1]}_{>0}f(x)\dd x\\
    &>0,
  \end{align*}
  using that $g(x,\cdot)$ is increasing for any fixed $x$
  (\cref{rem:density_aux_m}).
\end{proof}
\begin{figure}
    \centering\includegraphics[width=\linewidth]{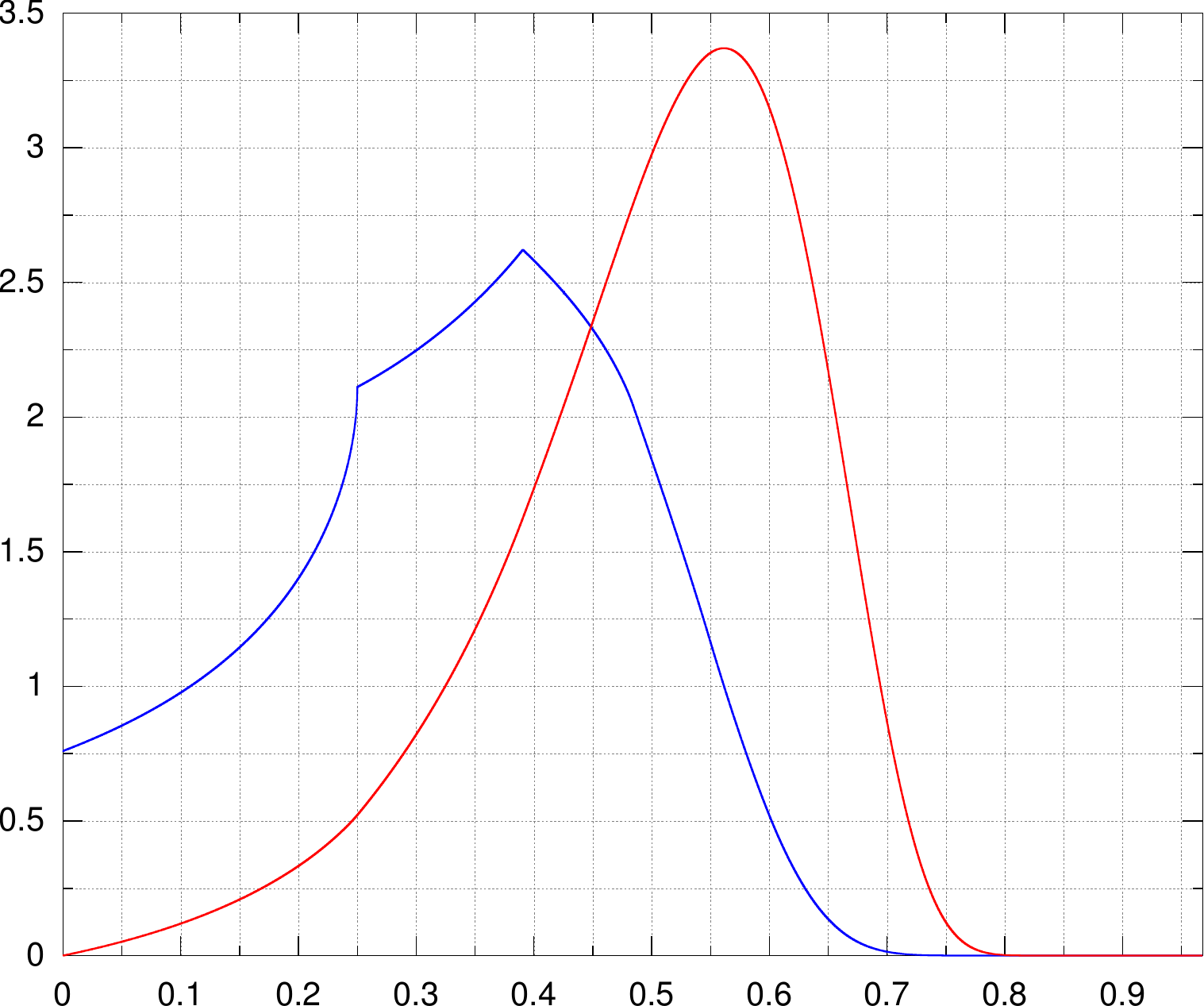}
    \caption{Approximated densities of RDE \eqref{rn_fix}
    (red) and RDE \eqref{rn_rde_2} (blue).}
   \label{fig:approximations}
   \vspace{-1em}
\end{figure}
\begin{theorem}\label{rn_den_bd}
  The density $f$ of $X$ in Theorem \ref{rn_thm1} is bounded with $\|f\|_\infty\leqslant109$.
\end{theorem}
\begin{proof}
  We bound $f(t)$ for $t\in(0,1)$ since $f(t)=0$ elsewhere.
 For  $t<\frac14$, using \eqref{eq:density} and the monotonicity in Remark \ref{rem:density_aux_m} we have the bound
  \begin{equation}\label{eq:density_first_bound}
    \begin{aligned}
    f(t)
    &\leqslant2\int_{p_t}^1g(x,t)f(x)\dd x
    \leqslant2g(0,t)\int_0^1f(x)\dd x\\
    &=\left(\frac14-t\right)^{-\frac12}.
    \end{aligned}
  \end{equation}
  Subsequently, we split the first integral in \eqref{eq:density_first_bound} into a left part where we will bound $f$
  and a right part where we will bound $g$: For any
  $\gamma\in(p_t,1]$, we split
  \begin{equation}
    f(t)\leqslant 2\int_{p_t}^\gamma g(x,t)f(x)\dd x
      +2\int_\gamma^1g(x,t)f(x)\dd x.%
    \label{eq:density_rude_bound}
  \end{equation}
  Let
  \begin{align*}
\gamma&=\gamma_t:=\frac{p_t+t}2\in(p_t,1],\\
  \mu_t&:=\sup\{f(\tau)\mid\tau\in(p_t,\gamma_t)\}.
  \end{align*}
  From
  \eqref{eq:density_rude_bound} we obtain, for all $t\in[0,1]$,
     \begin{align}
      f(t)&\leqslant 2\mu_t\int_{p_t}^{\gamma_t}g(x,t)\dd x
        +2g(\gamma_t,t)\int_{p_t}^1f(x)\dd x.\nonumber\\
      &=2\mu_tG(\gamma_t,t)\label{eq:density_main_bound}\\&\qquad
        +2g(\gamma_t,t)\Pb(X\geqslant1-2(1-\sqrt{t})),\nonumber
    \end{align}
  with $G(\cdot,\cdot)$ as in Remark \ref{rem:density_aux_m}. Hence
  \begin{equation}\label{eq:density_G}
  \begin{aligned}
    G(\gamma_t,t)=\frac12(1-\sqrt{t})\sqrt{1+6\sqrt{t}+t},%
  \end{aligned}
  \end{equation}
  and
  \begin{equation}
  \begin{aligned}
    g(\gamma_t,t)&=\frac{1+\frac{2\sqrt{t}-1+t}2}{G(\gamma_t,t)}\\
    &=\frac{(1+\sqrt{t})^2}{(1-\sqrt{t})\sqrt{1+6\sqrt{t}+t}}.%
    \label{eq:density_g}
  \end{aligned}
  \end{equation}
  We have $\gamma_\frac14=\frac18$, so, by \eqref{eq:density_first_bound},
  $\mu_\frac14\leqslant2\sqrt2$. Therefore
  \begin{align*}
    f\left(\frac14\right)
    &\leqslant4\sqrt2G\left(\frac18,\frac14\right)
      +2g\left(\frac18,\frac14\right)\Pb(X\geqslant0)\\
    &=\sqrt{\frac{17}2}+\frac{18}{\sqrt{17}}\leqslant8=:M_0.
  \end{align*}
 Since $f$ is increasing on $I_0:=[0,\frac14]$, see \cref{cor:density_increasing}, we obtain
 \begin{align*}
    f(t)\le M_0,\quad 0\le t\le \frac14.
   \end{align*}

  We now bound $f$ on $(\frac14,1)$. To do so we decompose this
  interval into subintervals $I_n$ where, for each $I_n$, we will deduce a bound $M_n$.
  Define $b_0:=0$, and, for $i\geqslant1$, $k\geqslant1$,
  \begin{gather*}
    b_i:=\left(\frac{b_{i-1}+1}2\right)^2,\\
    \begin{aligned}
    I_{2k-1}:=\left(b_k,\frac{b_k+b_{k+1}}2\right],\
    I_{2k}:=\left(\frac{b_k+b_{k+1}}2,b_{k+1}\right].
    \end{aligned}
  \end{gather*}
  We have $b_1=\frac14$ and  $b_i$ increases
  towards $1$ as $i\to\infty$, so that
  $$\left(\frac14,1\right)=\bigcup_{n=1}^{\infty}I_n.$$
 Let $I_{-1}:=\emptyset$. In a first step we show  for all $n\geqslant 1$ that \begin{align} \label{claim1_rn}
  (p_t,\gamma_t)\subseteq I_{n-2}\cup I_{n-1} \;\text{ for all }\;
   t\in I_n.
\end{align}
We denote $I_n=:(\alpha_n,\beta_n]$.
  If $n=2k-1,k\geqslant1,$ then
  $p_t>p_{\alpha_n}=p_{b_k}=b_{k-1}$, and
  $\gamma_t\leqslant\gamma_{\beta_n}\leqslant b_k$
  since
  $$\gamma_{\beta_n}
    :=\frac{2\sqrt{\frac{b_k+b_{k+1}}2}-1+\frac{b_k+b_{k+1}}2}2
    \leqslant b_k$$
  since $(17-b_k)(1-b_k)^3\geqslant0$.
  Hence
  $(p_t,\gamma_t)\subseteq(b_{k-1},b_k]=I_{2k-3}\cup I_{2k-2}$.
  In the other case $n=2k,k\geqslant1,$ we have
  $p_t\leqslant p_{\beta_n}=p_{b_{k+1}}=b_k$,
  so $\gamma_t:=\frac{p_t+t}2\leqslant\frac{b_k+b_{k+1}}2$, and
  $p_t>p_{\alpha_n}\geqslant\frac{b_{k-1}+b_k}2$ since
  $$p_{\alpha_n}:=2\sqrt{\frac{b_k+b_{k+1}}2}-1\geqslant
    \frac{b_{k-1}+b_k}2,$$
  which holds because of $(b_{k-1}-1)^4\geqslant0$.
  Thus
  $(p_t,\gamma_t)\subseteq I_{2k-2}\cup I_{2_k-1}$, and \eqref{claim1_rn} is
  proved.

Inductively we now define bounds $M_n$ for $f$ on $I_n$ for all $n\geqslant 0$. We already have $M_0=8$ and set  $M_{-1}:=0$. For each $n\geqslant1$ we use \eqref{eq:tail_probability} with
  $\varepsilon=2(1-\sqrt{t})$ and $k=2$, and obtain
  $$\Pb\left(X\geqslant1-2(1-\sqrt{t})\right)
    \leqslant8\sqrt2(1-\sqrt{t}).$$
  Plugging this into \eqref{eq:density_main_bound}, and
  substituting expressions \eqref{eq:density_G} and
  \eqref{eq:density_g}, we have, for all $t\in I_n$,
  \begin{align*}
    f(t)
    &\leqslant(1-\sqrt{t})\sqrt{1+6\sqrt{t}+t}\max\{M_{n-2},M_{n-1}\}\\
      &\qquad+\frac{16\sqrt2(1+\sqrt{t})^2}{\sqrt{1+6\sqrt{t}+t}}\\
    &\leqslant\left\lceil v(\alpha_n)\max\{M_{n-1},M_{n-2}\}\right\rceil
      +32=:M_n
  \end{align*}
  since the map $t\mapsto v(t):=(1-\sqrt{t})\sqrt{1+6\sqrt{t}+t}$
  is decreasing on $(\frac14,1)$.
  We obtain $M_0=8$, $M_1=41$, $M_2=71$, $M_3=93$, $M_4=106$, $M_5=109$,
  $M_6=106$, and since $v(t)<\frac{77}{109}$ for $t>b_4$,
  $M_n\leqslant109$ for $n>6$. This completes the proof of Theorem \ref{rn_den_bd}.
\end{proof}
The bound of $109$ in \cref{rn_den_bd} appears to be
poor as the plot in \cref{fig:approximations}  indicates
$\|f\|_\infty\leqslant 3.5$.
\begin{theorem} \label{rn_thm_rd}
  The version of the density $f$ of $X$ with \eqref{eq:density} has
  a right derivative at $0$ with
  $$f'_r(0)=\Erw{\frac2{(1+X)^2}}\approx 0.911364.$$
  Hence, $f$ is not differentiable at $0$, for $f'_\ell(0)=0$.
  We have $\E[X^{-2+\varepsilon}]<\infty$ for all $\varepsilon>0$.
\end{theorem}
\begin{proof}
  Let $t\in(0,\frac14]$. From \eqref{eq:density} we have
  \begin{equation}
  \begin{aligned}
         f(t)&=\int_0^12\I_{\{x<t\}}g(x,t)f(x)\dd x \\
      &\qquad+\int_0^1\I_{\{x>t\}}(g(x,t)-1)f(x)\dd x
   \end{aligned}
   \label{rn_fi}
  \end{equation}
where $g(\cdot,\cdot)$ is given in \eqref{eq:density_aux2}. For $x\in(0,1]$ we have
 \begin{align*}
    g(x,t)\leqslant\sqrt{\frac2x}
    =:\varphi(x).
  \end{align*}
  Hence $0\leqslant g(x,t)f(x)\leqslant\|f\|_\infty\varphi(x)$ for all
  $(x,t)\in(0,1]\times(0,\frac14]$.
  Since $g(x,t)\to1$ as $t\downarrow 0$ for all $x\in(0,1]$ and  $\varphi$ is integrable (on $(0,1]$) Lebesgue's
  dominated convergence theorem allows to interchange integration with the limit  $t\downarrow 0$. This implies
  $f(t)\to0$ as $t\downarrow 0$, thus $f$ is continuous at $0$.

  Now, substituting $x$ with $xt$ in the first integral in \eqref{rn_fi}, we obtain
  \begin{equation}
  \begin{aligned}
    \frac{f(t)}{\sqrt{t}}
    &=\int_0^12\sqrt{t}g(xt,t)f(xt)\dd x \label{rn_fi2}\\
      &\qquad+\int_0^1\I_{\{x>t\}}\frac{g(x,t)-1}{\sqrt{t}}f(x)\dd x
  \end{aligned}
  \end{equation}
  for all $t\in(0,\frac14]$. The first integrand in the latter display tends to $0$ as
  $t\downarrow 0$ and, using that $f$ is increasing on $(0,\frac14]$, see
  \cref{cor:density_increasing}, we obtain for all $x\in(0,1]$ that
  \begin{align*}
    0\leqslant\sqrt{t}g(xt,t)f(xt)\leqslant\frac{2f(xt)}{\sqrt{x}}
    \leqslant\frac{2f(x)}{\sqrt{x}}=:\psi(x).
  \end{align*}
 Note that $\psi$ is integrable since, using \eqref{rn_fix},
 \begin{align*}
    \int_0^1 \psi(x) \dd x &= 2\Erw{\frac1{\sqrt{X}}}\\
    &=2\Erw{\left(\sqrt{U}X+\sqrt{U}(1-\sqrt{U})\right)^{-\sfrac12}
      }\\
    &\leqslant2\Erw{\left(\sqrt{U}(1-\sqrt{U})\right)^{-\sfrac12}
      }\\
      &=2\pi.
  \end{align*}
  Hence, by dominated convergence, the first integrand in \eqref{rn_fi2} tends to $0$ as $t\downarrow 0$.
 For
  the second integrand in \eqref{rn_fi2}, plugging in \eqref{eq:density_aux2}, we find
\begin{align*}
    \frac{g(x,t)-1}{\sqrt{t}} \to 0 \;\text{ as } \; t\downarrow 0
     \end{align*}
  and  this fraction is dominated by $\varphi$  uniformly in $t\in(0,\frac14]$. Hence, altogether we obtain $f(t)/\sqrt{t}\to 0$
 as
  $t\downarrow 0$. In particular, $f(t)/\sqrt{t}$  is bounded by some constant $C$.
  Finally,
  \begin{align*}
    \frac{f(t)}t
    &=\int_0^12g(xt,t)f(xt)\dd x\\&\qquad
      +\int_0^1\I_{\{x>t\}}\frac{g(x,t)-1}tf(x)\dd x
  \end{align*}
  where the first integrand is dominated by $\sqrt8C$ and tends to $0$
  as $t\downarrow 0$ (since $f(xt)\to0$). The second integrand is
  dominated by $2\|f\|_\infty\varphi$ and tends to $2f(x)/(1+x)^2$ as $t\downarrow 0$.
  With the limit $t\downarrow 0$ and dominated convergence we obtain  that $f$ has a right derivative at $0$ with
  \begin{equation}
 \begin{aligned}
 f'_r(0)&=\Erw{\frac2{(1+X)^2}}\nonumber\\
  &=2\sum_{k=0}^{\infty}(-1)^k(k+1)\E[X^k].\label{re_ldrep}
    \end{aligned}
 \end{equation}
The interchange of summation and expectation in the latter display is
  justified by the fact that,
  for $0<\eta<1$,
\begin{align*}
\lefteqn{\left|\int_\eta^1\sum_{k=n+1}^{\infty}(-1)^k(k+1)x^kf(x)\dd x\right|}\\
  &\leqslant\int_\eta^1\sum_{k=n+1}^{\infty}(k+1)x^kf(x)\dd x\\
  &=\sum_{k=n+1}^{\infty}(k+1)\Pb(X^k\geqslant\eta),
  \end{align*}
where we used Levi's monotone convergence theorem and this is further bounded using Lemma \ref{rn_tailb} and denoting $\lambda:=-\log\eta$ by
\begin{align*}
  \lefteqn{\sum_{k=n+1}^{\infty}(k+1)\Pb(\textstyle
    X\geqslant1-\frac\lambda{k})}\\
  &\leqslant2^{10}\lambda^{\sfrac52}\sum_{k=n+1}^{\infty}(k+1)
    k^{-\sfrac52} \\
  &\to 0, \; \text{ as } \; n\to\infty
  \end{align*}
  and the series $\sum(-1)^k(k+1)f(x)x^k$ is normally convergent on
  $[0,\eta]$.

  The approximation for $f'_r(0)$ in the statement of Theorem \ref{rn_thm_rd} is obtained using \eqref{re_ldrep} and Lemma \ref{rn_lemmom}.

  Finally, since $t\mapsto f(t)/t$ remains bounded we obtain $\Erw{X^{-2+\varepsilon}}<\infty$ for all $\varepsilon>0$.
\end{proof}
\begin{theorem}
 For all $0<\varepsilon<1$, the version of the density $f$  with
 \eqref{eq:density} is H\"older continuous on $[0,1-\varepsilon]$
 with H\"older exponent $\frac12$:
 if $0\leqslant s<t\leqslant1-\varepsilon$, then
  $$|f(t)-f(s)|\leqslant(9+6\varepsilon^{-\sfrac32})\|f\|_\infty
    \sqrt{t-s}.$$
\end{theorem}
\begin{proof}
  Let $0\leqslant s<t\leqslant 1$. From the integral equation
  \eqref{eq:density}, we deduce that
\begin{align*}
 \lefteqn{ |f(t)-f(s)|}\\
 &\leqslant2\left|\int_{p_t}^tg(x,t)f(x)\dd x
    -\int_{p_s}^sg(x,s)f(x)\dd x\right|\\
    &\quad+
    \left|\int_t^1g(x,t)f(x)\dd x-\int_s^1g(x,s)f(x)\dd x\right|\\
    &\quad+\int_s^tf(x)\dd x\\
    &=:C_1+C_2+C_3.
\end{align*}
  We have $C_3\leqslant\|f\|_\infty(t-s)\leqslant\|f\|_\infty\sqrt{t-s}$.
  Using the primitive of $g(\cdot,t)$ given in \cref{rem:density_aux_m} and the monotonicity of $g(x,\cdot)$,
\begin{align*}
  C_1&\leqslant2\int_{p_t}^t(g(x,t)-g(x,s))
    f(x)\dd x\\
    &\quad+2\int_s^tg(x,s)f(x)\dd x+2\int_{p_s}^{p_t}g(x,s)f(x)\dd x\\
  &\leqslant2\|f\|_\infty\Biggl(\int_{p_t}^tg(\cdot,t)
    +\int_{p_s}^{p_t}g(\cdot,s)
    -\int_{p_t}^sg(\cdot,s)\Biggr)\\
  &\leqslant2\|f\|_\infty(4\sqrt{t-s}-(t-s))\\
  &\leqslant8\|f\|_\infty\sqrt{t-s}.
\end{align*}
Finally, with
$u_{x,s}:=\sqrt{(1+x)^2-4s}\geqslant u_{x,t}\geqslant\sqrt{\varepsilon}$
for all $t\leqslant x\leqslant1$, and using that $g(\cdot,s)$ is
decreasing,
\begin{align*}
  C_2&\leqslant\int_t^1(g(x,t)-g(x,s))f(x)\dd x+\int_s^tg(x,s)f(x)\dd x\\
  &=\int_t^1\frac{4(1+x)(t-s)f(x)}{u_{x,t}u_{x,s}(u_{x,t}+u_{x,s})}\dd x
    +\int_s^tg(x,s)f(x)\dd x\\
  &\leqslant\|f\|_\infty(t-s)\Biggl(\int_t^14\varepsilon^{-\sfrac32}
    \dd x+g(s,s)\Biggr)\\
  &\leqslant(4\varepsilon^{-\sfrac32}
    +2\varepsilon^{-1})\|f\|_\infty(t-s)\\
  &\leqslant6\varepsilon^{-\sfrac32}\|f\|_\infty\sqrt{t-s}.
\end{align*}
This completes the proof.
\end{proof}

For the distribution function of the limit $X$  in Theorem \ref{rn_thm1} we can apply a variant of a numerical
approximation developed in \cite{knne08} for which a rigorous error
analysis shows all values in the table of \cref{fig:cdf} being exact
up to $10^{-4}$.

\section{Perfect simulation}\label{sec_perf}
We construct an algorithm for perfect (exact) simulation from the limit
$X$ in \cref{rn_thm1}. We assume that a sequence of independent
and uniformly on $[0,1]$ distributed random variables is available
and that elementary operations of and between real numbers can be
performed exactly; see Devroye \cite{de86} for a comprehensive account
on non-uniform random number generation. Methods based on coupling from
the past have been developed and applied for the exact simulation from
perpetuities in \cite{fihu10,defa10,deja11,knne13,blsi11}; see also \cite{de01}. Our perpetuity
$X=_d\sqrt{U}X+\sqrt{U}(1-\sqrt{U})$ shares properties of $X=_d
UX+U(1-U)$ considered in \cite{knne13} which simplify the construction
of an exact simulation algorithm considerably compared to the examples
of the Vervaat perpetuities and the Dickman distribution in Fill and
Huber \cite{fihu10} and Devroye and Fawzi \cite{defa10}. Most notably
the  Markov chain underlying $X=_d\sqrt{U}X+\sqrt{U}(1-\sqrt{U})$
is positive Harris recurrent which allows to directly construct a
multigamma coupler as developed in  Murdoch and Green \cite[Section
2.1]{mugr98}. The design of the following algorithm {\tt
Simulate[$X=_d\sqrt{U}X+\sqrt{U}(1-\sqrt{U})$]} is similar to the
construction in \cite{knne13}: We construct an update function $\Phi: [0,1]\times \{0,1\}\times [0,1]\to [0,1]$
such that first for all $x\in [0,\infty)$ we have that $\sqrt{U}x+\sqrt{U}(1-\sqrt{U})$ and $\Phi(x,B,U)$ are identically distributed, where $U$ is uniformly distributed on $[0,1]$ and $B$ is an independent Bernoulli distributed random variable, and second  coalescence of the underlying Markov chains is supported.

 Recall  the densities $\varphi(x,\cdot)$ of $\sqrt{U}x+\sqrt{U}(1-\sqrt{U})$ given explicitly in \eqref{rn_den_sim}.
Fix $t\in(\frac18,\frac14)$. For all
$x\in[0,t]$, we have
$\varphi(x,t)\geqslant\varphi(t,t)=\frac{2(1+t)}{1-t}
  \geqslant\varphi(\frac18,\frac18)$,
and for $x\in(t,1]$, we obtain
$\varphi(x,t)\geqslant\varphi(1,t)=\frac1{\sqrt{1-t}}-1
  \geqslant\varphi(1,\frac18)\geqslant\varphi(\frac18,\frac18)$. Thus,
noting
$\alpha:=\varphi(\frac18,\frac18)=\sqrt{\sfrac87}-1\approx0{.}069$,
$$\varphi(x,t)\geqslant r(t):=\alpha\I_{(\frac18,\frac14)}(t)
  \quad\text{for all $(x,t)\in[0,1]^2$}.$$
Consequently,  we can write $\varphi(x,\cdot)=r+g_x$ for some nonnegative functions
$g_x$ for all $x\in[0,1]$. Note that $1=\|r\|_1+\|g_x\|_1$, with
$\|r\|_1:=\int_\R r(t)\dd t=\frac\alpha8$.

Let $R$, $Y^x$, $B$ be random variables with $R$ having density $r/\|r\|_1$, $Y^x$ having density
 $g_x/\|g_x\|_1$, and $B$ with $\text{Bernoulli}(\|r\|_1)$ distribution and   independent of $(R,Y^x)$. Then we have
$$BR+(1-B)Y^x\eqd\sqrt{U}x+\sqrt{U}(1-\sqrt{U}).$$
Hence we can use the update function
\begin{align*}
\Phi(x,b,u)&=b\left(\frac18u+\frac18\right)+(1-b)G_x^{-1}(u).%
\end{align*}
We construct our Markov chains from the past using $\Phi$ as an update function. In each
transition there is a probability of $\|r\|_1=\alpha/8$ that all chains couple simultaneously. In
other words, we can just start at a geometric Geom$(\alpha/8)$ distributed time $\tau$ in the past,
the first instant of $\{B = 1\}$ when moving back into the past. At this time $-\tau$ we
couple all chains via $X_{-\tau} :=\frac18U+\frac18$ and let the chains run from there until time $0$
using the updates $G_{X_{-k}}^{-1}(U_{-k+1})$ for $-k=-\tau,\ldots,-1$.
 It is shown in Murdoch and Green \cite[Section 2.1]{mugr98} that this is a valid implementation of the coupling from the past
algorithm in general.

Hence, it remains to invert the
distribution functions $G_x\colon[0,1]\to[0,1]$ of $Y^x$. We have
$\|g_x\|_1=1-\|r\|_1=\frac{8-\alpha}8$, and
\begin{align*}
  G_x(t)&:=\frac8{8-\alpha}\int_0^t(\varphi_x(u)-r(u))\dd u\\
  &=\frac8{8-\alpha}\left(F_x(t)-\frac\alpha8
    {\textstyle\max\{0,\min\{t-\frac18,\frac18\}\}}\right),
\end{align*}
with $F_x$ obtained in the proof of Theorem \ref{thm:density}.
The inversions of the functions $G_x$ can be computed explicitly and
lead to the functions $G^{-1}_x$ stated below.

With the sequence $(U_{-k})_{k\geqslant 0}$
of independent uniformly on $[0,1]$ distributed random variables and an independent
$\text{Geom}(\frac{\alpha}{8})$  geometrically distributed random variable we obtain the following algorithm:
\begingroup
\makeatletter
\let\@latex@error\@gobble
\makeatother
\bigskip
\begin{algorithm2e}[H]
\dontprintsemicolon
\SetKw{KwFrom}{from}
$\tau\gets\text{Geom}\left(\frac{1}{2\sqrt{14}}-\frac18\right)$\;
$X\gets\frac18U_{-\tau}+\frac18$\;
\BlankLine
\For{$k$ \KwFrom $-\tau+1$ \KwTo $0$}{
  $X\gets G_{X}^{-1}(U_k)$\;
}
\Return{$X$}\;
\caption{{\small\tt
Simulate[$X\eqd\sqrt{U}X+\sqrt{U}(1-\sqrt{U})$]}}%
\label{alg:sampling}
\end{algorithm2e}
\bigskip
\endgroup%
The function $G^{-1}$ is given by
\begin{equation*}
G_x^{-1}(u)=\begin{cases}
  _1G_x^{-1}(u),&\text{if $x\in[0,\sfrac18),u\in[0,\mathrm{a}_x)$},\\
  _2G_x^{-1}(u),&\text{if $x\in[0,\sfrac18),
    u\in[\mathrm{a}_x,\mathrm{b}_x)$},\\
  _3G_x^{-1}(u),&\text{if $x\in[0,\sfrac18),
    u\in[\mathrm{b}_x,\mathrm{c}_x)$},\\
  _4G_x^{-1}(u),&\text{if $x\in[0,\sfrac18),
    u\in[\mathrm{c}_x,1]$},\\
  _1G_x^{-1}(u),&\text{if $x\in[\sfrac18,\sfrac14),
    u\in[0,\mathrm{d}_x)$},\\
  _5G_x^{-1}(u),&\text{if $x\in[\sfrac18,\sfrac14),
    u\in[\mathrm{d}_x,\mathrm{e}_x)$},\\
  _3G_x^{-1}(u),&\text{if $x\in[\sfrac18,\sfrac14),
    u\in[\mathrm{e}_x,\mathrm{c}_x)$},\\
  _4G_x^{-1}(u),&\text{if $x\in[\sfrac18,\sfrac14),
    u\in[\mathrm{c}_x,1]$},\\
  _1G_x^{-1}(u),&\text{if $x\in[\sfrac14,1],
    u\in[0,\mathrm{d}_x)$},\\
  _5G_x^{-1}(u),&\text{if $x\in[\sfrac14,1],
    u\in[\mathrm{d}_x,\mathrm{f}_x)$},\\
  _6G_x^{-1}(u),&\text{if $x\in[\sfrac14,1],
    u\in[\mathrm{f}_x,\mathrm{g}_x)$},\\
  _4G_x^{-1}(u),&\text{if $x\in[\sfrac14,1],
    u\in[\mathrm{g}_x,1]$},\\
\end{cases}
\end{equation*}
where
\ifquantilefancyformat%
\begin{align*}
_1G_x^{-1}(u)&:=\left(\frac\alpha8-1\right)u
  +\frac{1+x}4\sqrt{2(8-a)}\sqrt{u},&&\\[0.5em]
_2G_x^{-1}(u)&:=\frac{64(1+x)^4-(8(1-u)+\alpha u)^2}{256(1+x)^2},&&\\
_3G_x^{-1}(u)&:=\frac1{8\alpha^2}\Bigl[2\alpha^2+\alpha(8-\alpha)
  (1-u)-16(1+x)^2\\&\hspace{-2.5em}+4(1+x)\sqrt{4(4+\alpha^2)
    (1+x)^2-2\alpha(8-\alpha)(1-u)-4\alpha^2}\Bigr],\\[0.5em]
_4G_x^{-1}(u)&:=\frac{64(1+x)^4-(8-\alpha)^2(1-u)^2}{256(1+x)^2},\\[0.5em]
_5G_x^{-1}(u)&:=\frac1{8(1+\alpha)^2}\Bigl[4\alpha(1+x)^2+(1+\alpha)
  (\alpha u+\alpha -8u)\\&\hspace{-1.5em}+2(1+x)
    \sqrt{4\alpha^2(1+x)^2-2(1+\alpha)(\alpha u+\alpha-8u)}\Bigr],
  \\[0.5em]
_6G_x^{-1}(u)&:=\left(\frac\alpha8-1\right)u+\frac{1+x}4
  \sqrt2\sqrt{8u+\alpha(1-u)}-\frac\alpha8,
\end{align*}
and
\begin{gather*}
\mathrm{a}_x:=\frac{8x^2}{8-\alpha},\
\mathrm{b}_x:=\frac{4(2-(1+x)\sqrt{4x^2+8x+2})}{8-\alpha},\\
\mathrm{c}_x:=\frac{8(1-(1+x)\sqrt{x^2+2x})-\alpha}{8-\alpha},\\[0.5em]
\mathrm{d}_x:=\frac{(2+2x-\sqrt{4x^2+8x+2})^2}{16-2\alpha},\
\mathrm{e}_x:=\frac{8x^2+(1-8x)\alpha}{8-\alpha},\\[0.5em]
\mathrm{f}_x:=\frac{4x^2+8x+2-\alpha-4(1+x)\sqrt{x^2+2x}}{8-\alpha},\\
\mathrm{g}_x:=\frac{8x^2-\alpha}{8-\alpha}.
\end{gather*}
Copyable versions of the latter expressions are given
below (
{\tt G}$k$ denotes $_kG_x^{-1}$ for
$k=1,\ldots,6$ and {\tt a}, {\tt b}, {\tt c}, {\tt d}, {\tt e},
{\tt f}, {\tt g} respectively denote $\mathrm{a}_x$, $\mathrm{b}_x$,
$\mathrm{c}_x$, $\mathrm{d}_x$, $\mathrm{e}_x$, $\mathrm{f}_x$,
$\mathrm{g}_x$).
\fi
\begin{lstlisting}[breaklines,basicstyle=\footnotesize\ttfamily]
A=sqrt(8./7)-1
G1=(A/8-1)*u+(1+x)/4*sqrt(2*(8-A))*sqrt(u)
G2=(64*(1+x)^4-(8*(1-u)+A*u)^2)/(256*(1+x)^2)
G3=(2*A*A+A*(8-A)*(1-u)-16*(1+x)^2+4*(1+x)*sqrt(4*(4+A*A)*(1+x)^2-2*A*(8-A)*(1-u)-4*A*A))/(8*A*A)
G4=(64*(1+x)^4-(8-A)^2*(1-u)^2)/(256*(1+x)^2)
G5=(4*A*(1+x)^2+(1+A)*(A*u+A-8*u)+2*(1+x)*sqrt(4*A*A*(1+x)^2-2*(1+A)*(A*u+A-8*u)))/(8*(1+A)^2)
G6=(A/8-1)*u+(1+x)*sqrt(2*(8*u+A*(1-u)))/4-A/8

a=8*x*x/(8-A)
b=4*(2-(1+x)*sqrt(4*x*x+8*x+2))/(8-A)
c=(8*(1-(1+x)*sqrt(x*x+2*x))-A)/(8-A)
d=(2+2*x-sqrt(4*x*x+8*x+2))^2/(16-2*A)
e=(8*x*x+(1-8*x)*A)/(8-A)
f=(4*x*x+8*x+2-A-4*(1+x)*sqrt(x*x+2*x))/(8-A)
g=(8*x*x-A)/(8-A)
\end{lstlisting}

\begin{figure}
    \centering\includegraphics[width=\linewidth]{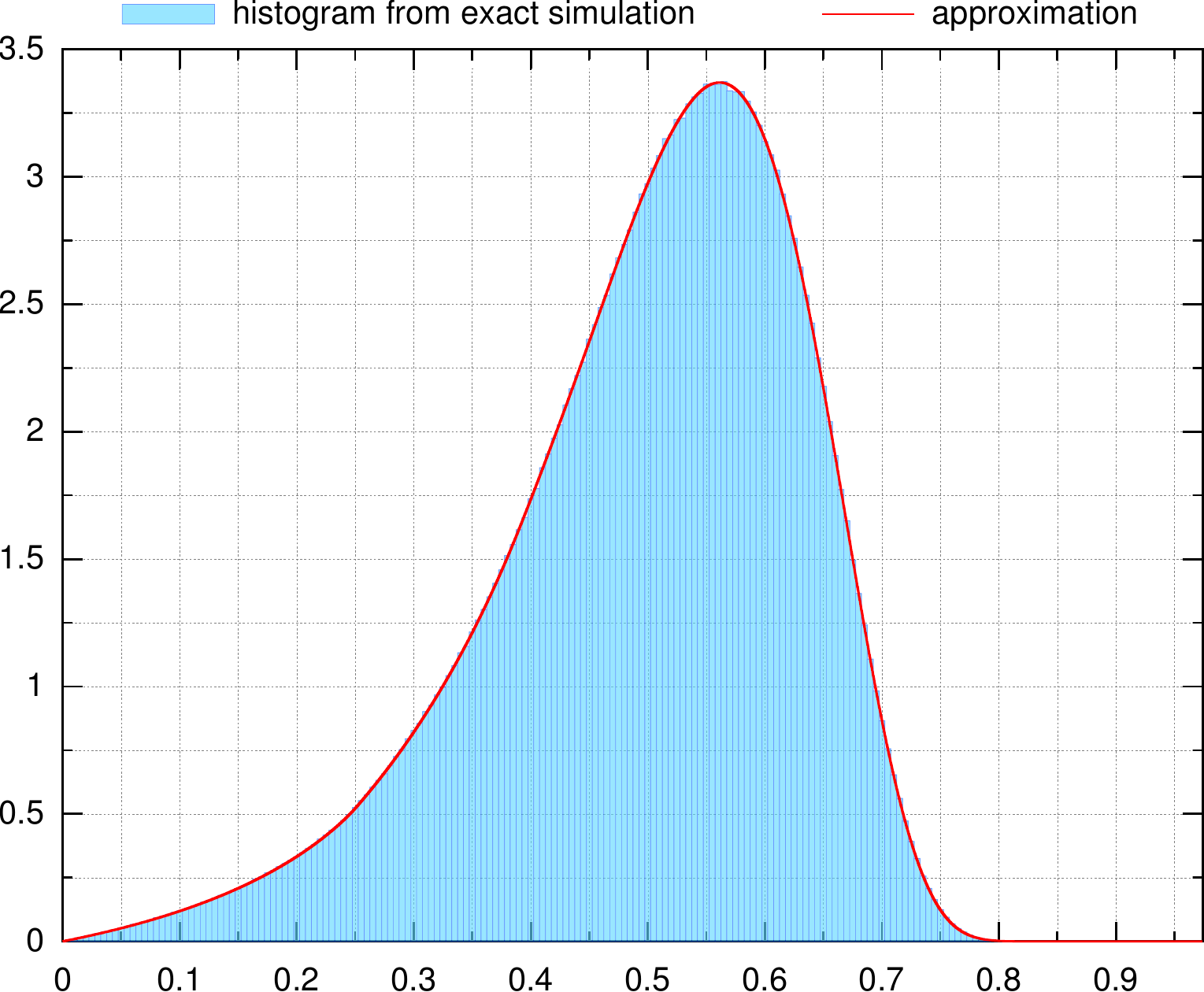}
    \caption{Normalized histogram of exact simulations  ($10{,}000{,}000$
    samples) of RDE \eqref{rn_fix} with Algorithm~\ref{alg:sampling}.}
    \label{fig:sampling}
\end{figure}
\FloatBarrier

\end{document}

%% file: cdf.tex
\begin{tabular}{|c||c|c|c|c|c|c|c|c|}
\hline
&0.0 & 0.1 & 0.2 & 0.3 & 0.4 & 0.5 & 0.6 & 0.7\\
\hline
\hline
0.000 & 0.0000 & 0.0054 & 0.0268 & 0.0811 & 0.2044 & 0.4400 & 0.7655 & 0.9768\\
\hline
0.005 & 0.0000 & 0.0060 & 0.0285 & 0.0853 & 0.2133 & 0.4550 & 0.7811 & 0.9809\\
\hline
0.010 & 0.0000 & 0.0067 & 0.0303 & 0.0896 & 0.2224 & 0.4703 & 0.7963 & 0.9844\\
\hline
0.015 & 0.0001 & 0.0074 & 0.0322 & 0.0942 & 0.2318 & 0.4858 & 0.8112 & 0.9874\\
\hline
0.020 & 0.0002 & 0.0081 & 0.0341 & 0.0989 & 0.2415 & 0.5016 & 0.8256 & 0.9900\\
\hline
0.025 & 0.0003 & 0.0089 & 0.0362 & 0.1038 & 0.2516 & 0.5175 & 0.8396 & 0.9922\\
\hline
0.030 & 0.0004 & 0.0097 & 0.0383 & 0.1089 & 0.2619 & 0.5337 & 0.8531 & 0.9939\\
\hline
0.035 & 0.0006 & 0.0106 & 0.0405 & 0.1142 & 0.2726 & 0.5500 & 0.8661 & 0.9954\\
\hline
0.040 & 0.0008 & 0.0115 & 0.0428 & 0.1198 & 0.2835 & 0.5665 & 0.8784 & 0.9965\\
\hline
0.045 & 0.0010 & 0.0125 & 0.0453 & 0.1255 & 0.2948 & 0.5831 & 0.8902 & 0.9975\\
\hline
0.050 & 0.0012 & 0.0135 & 0.0478 & 0.1314 & 0.3064 & 0.5999 & 0.9014 & 0.9982\\
\hline
0.055 & 0.0015 & 0.0146 & 0.0505 & 0.1376 & 0.3184 & 0.6167 & 0.9120 & 0.9987\\
\hline
0.060 & 0.0018 & 0.0157 & 0.0533 & 0.1440 & 0.3306 & 0.6335 & 0.9218 & 0.9991\\
\hline
0.065 & 0.0021 & 0.0169 & 0.0562 & 0.1506 & 0.3432 & 0.6503 & 0.9310 & 0.9994\\
\hline
0.070 & 0.0025 & 0.0181 & 0.0593 & 0.1575 & 0.3561 & 0.6672 & 0.9396 & 0.9996\\
\hline
0.075 & 0.0029 & 0.0194 & 0.0626 & 0.1647 & 0.3693 & 0.6839 & 0.9474 & 0.9997\\
\hline
0.080 & 0.0033 & 0.0208 & 0.0660 & 0.1721 & 0.3829 & 0.7006 & 0.9546 & 0.9998\\
\hline
0.085 & 0.0038 & 0.0222 & 0.0695 & 0.1797 & 0.3967 & 0.7171 & 0.9611 & 0.9999\\
\hline
0.090 & 0.0043 & 0.0237 & 0.0732 & 0.1877 & 0.4109 & 0.7335 & 0.9670 & 0.9999\\
\hline
0.095 & 0.0049 & 0.0252 & 0.0770 & 0.1959 & 0.4253 & 0.7496 & 0.9722 & 1.0000\\
\hline 
\end{tabular}